\documentclass[a4paper, 11pt, american]{amsart}

\usepackage[colorlinks,pdfpagelabels,pdfstartview = FitH,bookmarksopen
= true,bookmarksnumbered = true,linkcolor = blue,plainpages =
false,hypertexnames = false,citecolor = red,pagebackref=false]{hyperref}
\usepackage{times,latexsym,amssymb}
\usepackage{amsmath,amsthm,bm}
\usepackage{graphics, epsfig}
\usepackage{color}
\usepackage{appendix}
\usepackage[margin=1.5in]{geometry}

\usepackage{amsbsy}
\usepackage{amstext}
\usepackage{amssymb}
\usepackage{esint}
\usepackage{stmaryrd}

\renewcommand{\d}{\mathrm{d}}
\newcommand{\dx}{\mathrm{d}x}

\newcommand{\dt}{\mathrm{d}t}
\renewcommand{\rho}{\varrho}

\let\TeXchi\chi
\newbox\chibox
\setbox0 \hbox{\mathsurround0pt $\TeXchi$}
\setbox\chibox \hbox{\raise\dp0 \box 0 }
\def\chi{\copy\chibox}

\def\Xint#1{\mathchoice
    {\XXint\displaystyle\textstyle{#1}}%
    {\XXint\textstyle\scriptstyle{#1}}%
    {\XXint\scriptstyle\scriptscriptstyle{#1}}%
    {\XXint\scriptscriptstyle\scriptscriptstyle{#1}}%
    \!\int}
\def\XXint#1#2#3{\setbox0=\hbox{$#1{#2#3}{\int}$}
    \vcenter{\hbox{$#2#3$}}\kern-0.5\wd0}
\def\bint{\Xint-}
\def\dashint{\Xint{\raise4pt\hbox to7pt{\hrulefill}}}

\def\XXiint#1#2#3{\setbox0=\hbox{$#1{#2#3}{\iint}$}
    \vcenter{\hbox{$#2#3$}}\kern-0.5\wd0}

\author[N. Liao]{Naian Liao}
\address{Naian Liao,
Fachbereich Mathematik, Universit\"at Salzburg,
Hellbrunner Str. 34, 5020 Salzburg, Austria}
\email{naian.liao@sbg.ac.at}
\newtheorem{proposition}{Proposition}[section]
\newtheorem{theorem}{Theorem}[section]
\newtheorem{lemma}{Lemma}[section]

\newtheorem{remark}{Remark}[section]

\numberwithin{equation}{section}
\numberwithin{theorem}{section}
\numberwithin{proposition}{section}
\numberwithin{lemma}{section}
\numberwithin{remark}{section}
\newcommand{\noi}{\noindent}

\newcommand{\dsty}{\displaystyle}

\newcommand{\al}{\alpha}

\newcommand{\be}{\beta}

\newcommand{\Gm}{\Gamma}
\newcommand{\gm}{\gamma}
\newcommand{\dl}{\delta}
\newcommand{\Dl}{\Delta}
\newcommand{\lm}{\lambda}
\newcommand{\Lm}{\Lambda}

\newcommand{\varep}{\varepsilon}

\newcommand{\vp}{\varphi}
\newcommand{\sig}{\sigma}

\newcommand{\om}{\omega}

\newcommand{\z}{\zeta}

\newcommand{\nn}{\mathbb{N}}

\newcommand{\rr}{\mathbb{R}}
\newcommand{\rn}{\rr^N}

\newcommand{\bl}[1]{\mathbf{#1}}

\newcommand{\dvg}{\operatorname{div}}

\newcommand{\essup}{\operatornamewithlimits{ess\,sup}}
\newcommand{\essinf}{\operatornamewithlimits{ess\,inf}}
\newcommand{\essosc}{\operatornamewithlimits{ess\,osc}}

\newcommand{\loc}{\operatorname{loc}}

\newcommand{\dist}{\operatorname{dist}}

\newcommand{\pl}{\partial}

\newcommand{\intl}{\int\limits}

\def\Xint#1{\mathchoice
    {\XXint\displaystyle\textstyle{#1}}%
    {\XXint\textstyle\scriptstyle{#1}}%
    {\XXint\scriptstyle\scriptscriptstyle{#1}}%
    {\XXint\scriptscriptstyle\scriptscriptstyle{#1}}%
    \!\int}
\def\XXint#1#2#3{\setbox0=\hbox{$#1{#2#3}{\int}$}
    \vcenter{\hbox{$#2#3$}}\kern-0.5\wd0}
\def\bint{\Xint-}
\def\dashint{\Xint{\raise4pt\hbox to7pt{\hrulefill}}}
\def\dashiint{\bint\kern-0.15cm\bint}

\newcommand{\ovl}[3]{\int_{#1}^{#2}\kern-#3pt\raise4pt\hbox to7pt{\hrulefill}\ }

\newcommand{\ovll}[3]{\intl_{#1}^{#2}\kern-#3pt\raise4pt\hbox to7pt{\hrulefill}\ }

\newcommand{\tvl}[2]{\iint_{#1}\kern-#2pt\raise4pt\hbox to7pt{\hrulefill}\ }

\newcommand{\bye}{\end{document}}

\begin{document}
\title[Continuity for the Stefan problem]{Local continuity of weak solutions to the Stefan problem involving the singular $p$-Laplacian}
\date{}
\maketitle
\begin{abstract}
We establish the local continuity of locally bounded weak solutions (temperatures) to the doubly singular parabolic equation
modeling the phase transition of a material:
\[
\partial_t \beta(u)-\Delta_p u\ni 0\quad\text{ for }\tfrac{2N}{N+1}<p<2,
\]
where $\be$ is a maximal monotone graph with a jump at zero and $\Delta_p$ is the $p$-Laplacian.
Moreover, a logarithmic type modulus of continuity is quantified,
which has been conjectured to be optimal.

\end{abstract}
\section{Introduction}
The classical Stefan problem,
which describes the temperature of a material undergoing a phase change,
such as the melting of ice to water, gives rise to the following
nonlinear parabolic partial differential equation
\begin{equation}\label{Eq:1:0}
\pl_t\be(u)-\dvg\big(|Du|^{p-2}Du\big)\ni0\quad\text{ weakly in }E_T. 
\end{equation}
Here $E$ is an open set of $\rn$ with $N\ge1$, $E_T:=E\times(0,T]$ for some $T>0$,
and the enthalpy $\be(\cdot)$ is a maximal monotone graph in $\rr\times\rr$ defined by
\begin{equation}\label{Eq:beta}
\be(u)=\left\{
\begin{array}{cl}
u,\quad& u>0,\\[5pt] 
\left[-\nu,0\right],\quad& u=0,\\[5pt]
u-\nu,\quad& u<0,
\end{array}
\right.
\end{equation} 
for some positive constant $\nu$ reflecting the exchange of latent heat. 
The equation \eqref{Eq:1:0} is understood in the sense of differential inclusions,
which is to be made precise later.

Our {\bf main result} is that, when $\frac{2N}{N+1}<p<2$, locally bounded, weak solutions to \eqref{Eq:1:0} are locally continuous.
Moreover, we can quantify the modulus of continuity by
\begin{equation}\label{Eq:modulus}
(0,1)\ni r\mapsto \boldsymbol\om(r)=|\ln (cr)|^{-\sig}\quad\text{ for some }c,\,\sig>0.
\end{equation}
In fact, the same regularity holds true for more general diffusion part
 modeled on the $p$-Laplacian; see Theorem~\ref{Thm:1:1} for details.
The number $\frac{2N}{N+1}$ is referred to as a {\it critical} value
in the local regularity theory for the parabolic $p$-Laplace type equation,
and the given range of $p$ is often called {\it  singular, super-critical}; see \cite{DBGV-mono} for an account of the theory.

The local continuity of weak solutions to \eqref{Eq:1:0} is previously known only for $p\ge2$; 
see for instance \cite{Urbano-14, Caff-Evans-83, DB-82, Liao-Stefan, Urbano-00, Urbano-08}.
In particular, a modulus of type \eqref{Eq:modulus} has been achieved in \cite{Urbano-14} for all $p\ge2$,
which has been conjectured to be optimal as a structural property of weak solutions; see also \cite{Liao-Stefan}.
On the other hand, the case $p<2$ presents considerable difficulty. As far as we know, the only attempts
are made in \cite{Urbano-05, Li-20}. However, there is an oversight in \cite[Lemma~1]{Urbano-05},
which results in a gap in the proof of continuity; see \cite[p.~1382]{Li-20}. A stability result on the continuity issue
 as $p\uparrow2$ is presented in \cite{Li-20}.

Therefore our contribution to the existing literature is two-fold:
on the one hand, we confirm that locally bounded, weak solutions to \eqref{Eq:1:0} do have continuous representatives
when $\frac{2N}{N+1}<p<2$;
on the other hand, through more careful analysis,
 we explicitly quantify the logarithmic type modulus of continuity, which has been conjectured
to be optimal. The significance of continuity for weak solutions (temperatures) stems from their physical bearings;
an explicit and sharp modulus has further mathematical implications, cf.~\cite{Caff-Fried}.

Heuristically, the ``derivative" $\be'(0)$ is infinity and thus $\be(\cdot)$ is singular at $[u=0]$, while the $p$-Laplacian is singular
at $[|Du|=0]$. Hence the equation \eqref{Eq:1:0} carrys two types of singularities simultaneously
and the main difficulty lies in balancing them.
To give a glimpse at what is at stake here, suppose $u$ is a solution to  \eqref{Eq:1:0} and $0\le u\le 1$.
The goal is to show that the oscillation of $u$ is strictly less than $1$ within a smaller cylinder in a quantified manner.
The classical alternative approach (cf.~\cite{DB, Urbano-08}) unfolds along 
the smallness of the measure concentration of the set $[u\approx0]$.
If $|[u\approx0]|<c_o$ for some critical number $c_o$, then a De Giorgi type lemma asserts that
$u$ has pointwise concentration near $1$ in a smaller cylinder, that is,
 $u\ge\sig$ for some $\sig>0$. 
Thus the oscillation is controlled by $1-\sig$ in that cylinder. Otherwise, if $|[u\approx0]|\ge c_o$,
then the machinery of De Giorgi implies that $u$ has pointwise concentration near $0$ in a smaller cylinder, 
that is, $u\le 1-\sig$ for some $\sig>0$.
In either case, the oscillation is brought down by $\sig$ in a quantified smaller cylinder.
Iterating  this process and reducing the oscillation of $u$ along a nested family of cylinders will prove the continuity of $u$.

Nevertheless, due to the singularity of the equation \eqref{Eq:1:0}, we need to stretch or compress the cylinders
at each step, according to the oscillation itself over the nested cylinders; this is called the {\it method of intrinsic scaling}, 
cf.~\cite{DB,DBGV-mono, Urbano-08}. Moreover, propagation of measure information in the time direction,
connecting the two alternatives,
becomes quite delicate in this intrinsic scaling scenario.
As a result, we encounter the difficulty:
when we work near $u\approx0$ in the first alternative, the singularity of $\be(u)$ dominates that of the $p$-Laplacian
and a ``long", intrinsic cylinder is needed; 
when we work near $u\approx1$ in the second alternative, $\be(u)$ does not contribute any singularity,
only the singularity of the $p$-Laplacian
matters,  and a ``short", intrinsic cylinder is called for.
As such the two alternatives are potentially incompatible with each other, if the classical  approach
is not properly implemented, in view of the delicacy of intrinsic scaling.

Our new device to overcome this difficulty consists in exploiting 
a Harnack type inequality in the $C_{\loc}\big(0,T; L^1_{\loc}(E)\big)$ topology, for non-negative, weak super-solutions to the parabolic $p$-Laplacian (which is termed an $L^1$ Harnack inequality in Section~\ref{S:L1}),
in order to propagate the measure information along the time direction and thus to bridge the two alternatives.
Together with a tailored expansion of positivity (cf.~Section~\ref{S:expansion}), 
the measure information is then translated into pointwise estimates, 
yielding the reduction of oscillation. The idea of using an ameliorated expansion of positivity
to obtain a logarithmic type modulus of continuity for the Stefan problem stems from \cite{Liao-Stefan}.
The use of an $L^1$ Harnack inequality has been inspired by \cite{Vespri-ACV} with a different context.
Since there are considerable technicalities to follow,  
it will be helpful to refer to \cite[Section~1.3]{Liao-Stefan} to have an overview of the method.

\subsection{Statement of the  results}
%

In what follows, we consider the following
more general, parabolic equation
\begin{equation}\label{Eq:1:1}
\pl_t\be(u)-\dvg\bl{A}(x,t,u, Du) \ni 0\quad \text{ weakly in }\> E_T.
\end{equation}
Here $\be(\cdot)$ is defined in \eqref{Eq:beta}.   
The function $\bl{A}(x,t,u,\xi)\colon E_T\times\rr^{N+1}\to\rn$ is assumed to be
measurable with respect to $(x, t) \in E_T$ for all $(u,\xi)\in \rr\times\rn$,
and continuous with respect to $(u,\xi)$ for a.e.~$(x,t)\in E_T$.
Moreover, we assume the structure conditions
\begin{equation}  \label{Eq:1:2}
\left\{
\begin{array}{l}
\bl{A}(x,t,u,\xi)\cdot \xi\ge C_o|\xi|^p \\[5pt]
|\bl{A}(x,t,u,\xi)|\le C_1|\xi|^{p-1}%
\end{array}%
\right .\quad \text{ a.e.}\> (x,t)\in E_T,\, \forall\,u\in\rr,\,\forall\xi\in\rn,
\end{equation}
where $C_o$ and $C_1$ are given positive constants, and we take
\begin{equation}\label{Eq:p-range}
\tfrac{2N}{N+1}=:p_*<p<2.
\end{equation}
%
In the sequel, the set of parameters $\{\nu, p,N,C_o,C_1\}$ will be referred to as the (structural) data
and we will use $\gm$ as a generic positive constant depending on the data.

In addition, the 
principal part $\bl{A}$ is assumed to be monotone in the variable $\xi$ 
in the sense that
\begin{equation}\label{Eq:5:1:3}
\big(\bl{A}(x,t,u,\xi_1)-\bl{A}(x,t,u,\xi_2)\big)
\cdot(\xi_1-\xi_2)\ge0
\end{equation}
for all variables in the indicated domains, and Lipschitz continuous 
in the variable $u$, that is,
\begin{equation}\label{Eq:5:1:4}
\big|\bl{A}(x,t,u_1,\xi)-\bl{A}(x,t,u_2,\xi)\big|\le \Lm|u_1-u_2| (1+|\xi|^{p-1})
\end{equation}
for some given $\Lm>0$, and for the variables in the indicated domains.

Before proceeding, let us recall the parabolic $p$-Laplace type equation, which is pertinent to \eqref{Eq:1:1}:
\begin{equation}\label{Eq:p-Laplace}
u_t-\dvg\bl{A}(x,t,u, Du) = 0\quad \text{ weakly in }\> E_T,
\end{equation}
where the properties of $\bl{A}(x,t,u, \xi)$ are retained.


Let $\Gm:=\pl E_T-\overline{E}\times\{T\}$
be the parabolic boundary of $E_T$, and for a compact set $\mathcal{K}\subset E_T$
introduce the parabolic $p$-distance from $\mathcal{K}$ to $\Gm$ by
\begin{equation*}
	\begin{aligned}
		\dist_p(\mathcal{K};\,\Gm)&:=\inf_{\substack{(x,t)\in \mathcal{K}\\(y,s)\in\Gm}}
		\left\{|x-y|+|t-s|^{\frac1p}\right\}.
	\end{aligned}
\end{equation*}

The formal definition of weak solution to \eqref{Eq:1:1} will be postponed to Section~\ref{S:1:2}.
Now we proceed to present the regularity theorem.
\begin{theorem}\label{Thm:1:1}
Let $u$ be a bounded weak solution to \eqref{Eq:1:1} in $E_T$, 
 under the structure condition \eqref{Eq:1:2} for $p\in(p_*,2)$.  
Then $u$ is locally continuous in 
$E_T$.
More precisely, there is a modulus of continuity $\boldsymbol\om(\cdot)$,
determined by the data, $\dist_p(\mathcal{K};\,\Gm) $ and $\|u\|_{\infty,E_T}$, 
such that
	\begin{equation*}
	\big|u(x_1,t_1)-u(x_2,t_2)\big|
	\le
	\boldsymbol\om\!\left(|x_1-x_2|+|t_1-t_2|^{\frac1p}\right),
	\end{equation*}
for every pair of points $(x_1,t_1), (x_2,t_2)\in \mathcal{K}$.
If in addition, the structure conditions \eqref{Eq:5:1:3} -- \eqref{Eq:5:1:4}
are satisfied, then there exists $\sig\in(0,1)$ depending only on $p$ and $N$, such that
$$\boldsymbol\om(r)=C \Big(\ln \frac{\dist_p(\mathcal{K};\,\Gm)}{r}\Big)^{-\sig}\quad\text{ for all }r\in\big(0, \dist_p(\mathcal{K};\,\Gm)\big)$$ 
with some $C>0$
 depending on the data and $\|u\|_{\infty,E_T}$.
\end{theorem}
\begin{remark}\upshape
The number $\sig$ can be quantified by $\frac{N(p-2)+p}{4(N+p)}$.
\end{remark}
\begin{remark}\upshape
Without the additional structural conditions \eqref{Eq:5:1:3} -- \eqref{Eq:5:1:4}, a modulus
of continuity can also be quantified, which is however far from being optimal, cf.~Proposition~\ref{Prop:int-reg}.
The conditions \eqref{Eq:5:1:3} -- \eqref{Eq:5:1:4} have been evoked to validate the comparison principle (cf.~\cite[Chapter~7, Corollary~1.1]{DBGV-mono}),
which is needed in the second part of Lemma~\ref{Lm:expansion}.
Although we think they are just technical assumptions, we do not know how to remove them for now.
On the other hand, they are natural assumptions in the existence theory,
in order to construct continuous weak solutions; see Section~\ref{S:approx}.
\end{remark}
\begin{remark}\upshape
Local boundedness of weak solutions is sufficient in the theorem.
The method also applies to equations with
 lower order terms. The maximal monotone graph $\be(\cdot)$
also admits general forms, as long as it has only one jump; see \cite{DBV-95} for multiple jumps. 
The changes due to these generalizations can be modeled on the arguments in \cite{DB-82}.
However we will not pursue generality in this direction. Instead, concentration will be made on the actual novelties.
\end{remark}

\subsection{Definition of solution}\label{S:1:2}
Local weak solutions to the parabolic $p$-Laplace type equation \eqref{Eq:p-Laplace}
are defined in the function space
\begin{equation}   \label{Eq:func-space-1}
	u\in C_{\loc}\big(0,T;L_{\loc}^2(E)\big)\cap L_{\loc}^p\big(0,T; W_{\loc}^{1,p}(E)\big),
\end{equation}
which assumes no {\it a priori} knowledge on the time derivative; see \cite[Chapter~II]{DB}.

In contrast to \eqref{Eq:func-space-1}, the notion of weak solution  to the Stefan problem \eqref{Eq:1:1}
 requires the
 time derivative to exist in the Sobolev sense. This is necessary to justify the calculations to follow. 
On the other hand, 
construction of  {\it continuous} weak solutions to \eqref{Eq:1:1} without time derivate
will be indicated in Section~\ref{S:approx}.

A function $u$
is termed a local, weak sub(super)-solution to \eqref{Eq:1:1} with the structure
condition \eqref{Eq:1:2}, if 
\begin{equation*}
	u\in W_{\loc}^{1,2}\big(0,T;L^2_{\loc}(E)\big)\cap L^p_{\loc}\big(0,T; W^{1,p}_{\loc}(E)\big)
\end{equation*}
and if
for every compact set $K\subset E$ and every sub-interval
$[t_1,t_2]\subset (0,T]$, there is a selection $v\subset\be(u)$, i.e.
\[
\left\{\big(z,v(z)\big): z\in E_T\right\}\subset \left\{\big(z,\be[u(z)]\big): z\in E_T\right\},
\]
 such that
\begin{equation*}
	\int_K v\z \,\dx\bigg|_{t_1}^{t_2}
	+
	\iint_{K\times(t_1,t_2)} \big[-v\pl_t\z+\bl{A}(x,t,u,Du)\cdot D\z\big]\dx\dt
	\le(\ge)0
\end{equation*}
for all non-negative test functions
\begin{equation}\label{Eq:function-space}
\z\in W^{1,2}_{\loc}\big(0,T;L^2(K)\big)\cap L^p_{\loc}\big(0,T;W_o^{1,p}(K)%
\big).
\end{equation}
Since $v\in L^{\infty}_{\loc} \big(0,T;L^2_{\loc}(E)\big)$,  all the above integrals 
are well-defined. It is also noteworthy that the above integral formulation itself does not involve
the time derivative of $u$.
On the other hand, it is possible to employ the time derivative  and write the integral formulation as
\begin{equation}\label{Eq:int-form}
\begin{aligned}
-\int_K&\nu(x,t)\chi_{[u\le0]}\z\,\dx\bigg|_{t_1}^{t_2}+\iint_{K\times(t_1,t_2)}\nu(x,t)\chi_{[u\le0]}\pl_t\z\,\dx\dt\\
&+\iint_{K\times(t_1,t_2)} \big[\pl_t u\z+\bl{A}(x,t,u,Du)\cdot D\z\big]\dx\dt
	\le(\ge)0,
\end{aligned}
\end{equation}
where $\z$ is as in \eqref{Eq:function-space}, 
$\chi$ is the characteristic function of the indicated set,
and the non-negative measurable function $\nu(x,t)$ is given by
\begin{equation*}
\nu(x,t):=\left\{
\begin{array}{cl}
\nu,&\quad (x,t)\in [u<0],\\[5pt]
-v(x,t),&\quad (x,t)\in [u=0].
\end{array}\right.
\end{equation*}

A function $u$ is a local weak solution
to \eqref{Eq:1:1}, 
if it is both a local weak sub-solution and a local weak super-solution.

\

\noi {\it Acknowledgement.} This research has been funded by 
the FWF--Project P31956--N32 “Doubly nonlinear evolution equations”.
\section{Preliminaries}\label{S:2}
In this section we collect some preparatory materials, including
an energy estimate (Proposition~\ref{Prop:2:1}), a De Giorgi type lemma (Lemma~\ref{Lm:DG:int}),
the expansion of positivity  (Lemma~\ref{Lm:expansion})
and an $L^1$ Harnack inequality (Lemma~\ref{Lm:L1}) for non-negative, weak super-solutions to the parabolic $p$-Laplacian.

Throughout the rest of this note, 
we will use the symbol $K_\rho(x_o)$ to denote the cube of side length $2\rho$
and center $x_o$, with faces parallel with the coordinate planes of $\rr^N$.
Moreover, we will use
\begin{equation*}
\left\{
\begin{aligned}
Q_\rho(\theta)&:=K_{\rho}(x_o)\times(t_o-\theta\rho^p,t_o),\\[5pt]
Q_{R,S}&:=K_R(x_o)\times (t_o-S,t_o),
\end{aligned}\right.
\end{equation*} 
to denote (backward) cylinders with the indicated positive parameters; 
we omit the vertex $(x_o,t_o)$
from the notations for simplicity.
\subsection{Energy estimates}
The energy estimate, concerning the truncated functions
$(u-k)_+:=\max\{u-k,0\}$ and $(u-k)_-:=\max\{k-u,0\}$, is the same as \cite[Proposition~2.1]{Liao-Stefan}.
Needless to say, it holds true for all $p>1$.
\begin{proposition}\label{Prop:2:1}
	Let $u$ be a  local weak sub(super)-solution to \eqref{Eq:1:1} with  \eqref{Eq:1:2} in $E_T$.
	There exists a constant $\gm (C_o,C_1,p)>0$, such that
 	for all cylinders $Q_{R,S}\subset E_T$, 
 	every $k\in\rr$, and every non-negative, piecewise smooth cutoff function
 	$\z$ vanishing on $\pl K_{R}(x_o)\times(t_o-S,t_o)$,  there holds
\begin{align*}
	\essup_{t_o-S<t<t_o}&\Big\{\int_{K_R(x_o)\times\{t\}}	
	\z^p (u-k)_\pm^2\,\dx +\Phi_{\pm}(k, t_o-S,t,\z)\Big\}\\
	&\quad+
	\iint_{Q_{R,S}}\z^p|D(u-k)_\pm|^p\,\dx\dt\\
	&\le
	\gm\iint_{Q_{R,S}}
		\Big[
		(u-k)^{p}_\pm|D\z|^p + (u-k)_\pm^2|\partial_t\z^p|
		\Big]
		\,\dx\dt\\
	&\quad
	+\int_{K_R(x_o)\times \{t_o-S\}} \z^p (u-k)_\pm^2\,\dx,
\end{align*}
where 
\begin{align*}
\Phi_{\pm}(k, t_o-S,t,\z)=&-\int_{K_R(x_o)\times\{\tau\}}\nu(x,\tau)\chi_{[u\le0]}[\pm(u-k)_\pm\z^p]\,\dx\Big|_{t_o-S}^t\\
&+\int_{t_o-S}^t\int_{K_R(x_o)}\nu(x,\tau)\chi_{[u\le0]}\pl_t[\pm(u-k)_\pm\z^p]\,\dx\d\tau
\end{align*}
and $\nu(x,t)$ is a selection out of $[0,\nu]$ for $u(x,t)=0$, and $\nu(x,t)=\nu$ for $u(x,t)<0$.
\end{proposition}

\subsection{A De Giorgi type lemma}
For  a cylinder $Q_\rho=K_{\rho}(x_o)\times(t_o-\rho^{p},t_o)\subset E_T$ with $\rho\in(0,1)$,
we introduce the numbers $\mu^{\pm}$ and $\om$ satisfying
\begin{equation*}
	\mu^+\ge\essup_{Q_\rho} u,
	\quad 
	\mu^-\le\essinf_{Q_\rho} u,
	\quad
	\om\ge\mu^+ - \mu^-.
\end{equation*}
Let $\theta\in(0,1)$ be  a parameter to be determined.
The cylinder $Q_\rho(\theta)$ is coaxial with $Q_\rho$ and with the same vertex $(x_o,t_o)$;
moreover, we observe that $Q_\rho(\theta) \subset Q_\rho$.
Then we have the following lemma valid for all $p>1$; we refer to \cite[Lemma~2.1]{Liao-Stefan} for a proof.
\begin{lemma}\label{Lm:DG:int}
Let $u$ be a  local weak super-solution to \eqref{Eq:1:1} with \eqref{Eq:1:2} in $E_T$. For $\xi\in(0,1)$,
set $\theta=(\xi\om)^{2-p}$ and assume that $Q_{\rho}(\theta)\subset Q_\rho$. 
There exists a positive constant $c_o$ depending only on the data, such that if
\[
|[u\le\mu^-+\xi\om]\cap Q_{\rho}(\theta)|\le c_o(\xi\om)^{\frac{N+p}p}|Q_{\rho}(\theta)|,
\]
then
\[
u\ge \mu^-+\tfrac12\xi\om\quad\text{ a.e. in }Q_{\frac12\rho}(\theta).
\]
\end{lemma}
\begin{remark}\upshape
The singularity of $\be(\cdot)$ is reflected by $(\xi\om)^{\frac{N+p}p}$ in the measure condition.
An essential task in estimating the modulus of continuity lies in carefully tracing its effect at every step
and quantifying its role in the reduction of oscillation.
\end{remark}
\subsection{Expansion of Positivity}\label{S:expansion}
The following lemma concerns the expansion of positivity for non-negative, super-solutions to the parabolic $p$-Laplacian.
It consists of two parts: the first part actually holds for all $p\in(1,2)$, under the mere structural condition \eqref{Eq:1:2};
the second part provides an amelioration, yet under the additional conditions \eqref{Eq:p-range} -- \eqref{Eq:5:1:4}.
Although we think the second part should hold true under the same assumptions as the first part,
we do not have a proof for the moment. In any case, the current version suffices for our purpose.
When we speak of the {\it data} in Sections~\ref{S:expansion} -- \ref{S:L1}, it is referred to $\{p, N, C_o, C_1\}$
\begin{lemma}\label{Lm:expansion}
Let $v$ be a non-negative,  local, weak super-solution to \eqref{Eq:p-Laplace} with 
 \eqref{Eq:1:2} for $p\in(1,2)$ in $E_T$.
Suppose that for some $M,\,\al\in(0,1)$ there holds
\[
|[v(\cdot, t_o)>M]\cap K_{\rho}(x_o) |\ge\al |K_{\rho}|.
\]
Then there exist positive constants $\dl$ and $\eta$ depending only on the data and $\al$, 
such that
\[
v(\cdot, t)\ge\eta M\quad\text{ a.e. in } K_{ \rho}(x_o).
\]
for all times
\[
t_o+\tfrac12\dl M^{2-p}\rho^p\le t\le t_o+\dl M^{2-p}\rho^p,
\]
provided $\mathcal{Q}:=K_{8\rho}\times(t_o,t_o+\dl M^{2-p}\rho^p)$ is included in $E_T$.
Moreover, the constants $\dl$ and $\eta$ can be traced by
\[
\dl=\dl_o \al^{p+1},\quad\eta=\eta_o \exp \Big\{-2^{\gm/\al^{p+2}}\Big\}  
\]
for some  $\dl_o,\,\eta_o\in(0,1)$ and $\gm>1$ depending only on the data.

If in addition, 
the structure conditions \eqref{Eq:5:1:3} -- \eqref{Eq:5:1:4} are satisfied and $p\in(p_*,2)$,
then the dependence of $\eta$ can be improved to be
\[
\eta=\eta_o\al,
\]
for some $\eta_o\in(0,1)$ depending only on the data. 
\end{lemma}
\begin{proof}
The first part can be retrieved from Proposition~5.1 and Remark~5.1 in \cite[Chapter~4]{DBGV-mono}.
Thus we will only deal with the improvement of the dependence of $\eta$ under the additional
assumption 
that the structure conditions \eqref{Eq:5:1:3} -- \eqref{Eq:5:1:4} are verified.

To this end, let $(x_o,t_o)=(0,0)$ and set $Q=K_{8\rho}\times(0,\infty)$. 
Let $w$ be the unique, continuous, non-negative, weak solution to the following Cauchy-Dirichlet problem:
\begin{equation*}
\left\{
\begin{aligned}
&w_t-\dvg \bl{A}(x,t,w,Dw)=0\quad\text{ weakly in } Q,\\
&w(\cdot,t)\Big|_{\pl K_{8\rho}}=0\quad\text{ a.e. }t\in(0,\infty),\\
&w(\cdot,0)=v(\cdot,0)\chi_{K_{\rho}}.
\end{aligned}
\right.
\end{equation*}
Then, according to the given measure information of $v$, there holds 
\begin{equation}\label{Eq:measure-w}
|[w(\cdot,0)>M]\cap K_{\rho}|\ge\al |K_{\rho}|.
\end{equation}
Starting from \eqref{Eq:measure-w}, an application of \cite[Chapter~3, Lemma~1.1]{DBGV-mono} yields that there exist
\[
\varep=\tfrac18\al\quad\text{and}\quad\dl=\dl_o\al^{p+1},
\]
for some $\dl_o\in(0,1)$ depending only on the data,
such that
\[
|[w(\cdot,t)>\varep M]\cap K_{\rho}|\ge\tfrac12\al |K_{\rho}|,
\]
for all times
\[
0<t\le\dl M^{2-p}\rho^p=:t_1.
\]
Let us focus on the time level $t_1$. 
By the above measure information at the time level $t_1$
and the continuity of the function $x\mapsto w(x,t_1)$, there must be some $x_1\in K_{\rho}$
satisfying $w(x_1,t_1)\ge\varep M$.

Next, we may apply the intrinsic Harnack inequality in \cite[Chapter~7, Theorem~2.2]{DBGV-mono} at $w(x_1,t_1)$,
which requires that $p\in(p_*,2)$.
In this way, there exist constants $a,\,\eta_o\in(0,1)$ depending only on the data, such that
\[
\inf_{(x_1,t_1)+Q_{2\rho}(\theta)}w\ge\eta_o\varep M\quad\text{ where }\theta=a(\varep M)^{2-p}.
\]
Note that we can choose $a$ to be smaller and to depend on $\al$, as long as it is verified that
$a(\varep M)^{2-p}(8\rho)^p\le t_1$, that is, $a\le 8^{2-2p}\dl_o\al^{2p-1}$.
Therefore, recalling that $x_1\in K_\rho$, by the comparison principle (cf.~\cite[Chapter~7, Corollary~1.1]{DBGV-mono}) we have
\[
\inf_{K_{\rho}} v(\cdot, t)\ge \eta_o \varep M\equiv \tfrac18\eta_o\al M
\]
for all times
\[
t_1-a(\tfrac18\al M)^{2-p}(2\rho)^p\le t\le t_1.
\]
Recalling the definition of $t_1$ and properly adjusting $\{a, \dl_o, \eta_o\}$,
we can conclude.
\end{proof}
\subsection{An $L^1$ Harnack inequality}\label{S:L1}
The following can be retrieved from Proposition~A.1.2 in \cite[Appendix~A]{DBGV-mono}.
\begin{lemma}\label{Lm:L1} 
Suppose $v$ is a non-negative, local,
weak super-solution to \eqref{Eq:p-Laplace} with \eqref{Eq:1:2} in $E_T$ for $p\in(1,2)$.  
There exists a positive constant 
$\gm$ depending only on the data , such that 
for any cylinder $[s,t]\times K_{2\rho}(y)\subset E_T$,
\begin{equation*}
\sup_{s<\tau<t}\int_{K_\rho(y)} v(x,\tau)\,\dx\le\gm
\int_{K_{2\rho}(y)}v(x,t)\,\dx
+\gm\Big(\frac{t-s}{\rho^\lm}\Big)^{\frac1{2-p}}
\end{equation*}
where 
\begin{equation*}
\lm:=N(p-2)+p.
\end{equation*}
\end{lemma}
\begin{remark}\upshape
Lemma~\ref{Lm:L1} holds true for all $p\in(1,2)$ and the conditions \eqref{Eq:5:1:3} -- \eqref{Eq:5:1:4} are not needed.
The super-critical range \eqref{Eq:p-range} of $p$ is actually equivalent to $\lm\in(0,p)$.
 Hence Lemma~\ref{Lm:L1} holds true irrespective of the signs of $\lm$.
 However, the constant $\gm=\gm(p)\to\infty$ either as $p\uparrow2$ or as $p\downarrow1$. 
\end{remark}
\subsection{An auxiliary lemma}
The next lemma asserts that if we truncate a sub-solution to the Stefan problem \eqref{Eq:1:1} by a positive number,
the resultant function is a sub-solution to the parabolic $p$-Laplacian \eqref{Eq:p-Laplace}.
This actually holds for all $p>1$.
\begin{lemma}\label{Lm:sub-solution}
Let $u$ be a  local weak sub-solution to \eqref{Eq:1:1} with  \eqref{Eq:1:2} in $E_T$. 
Then for any $k>0$, the truncation $(u-k)_+$ is a local weak sub-solution to \eqref{Eq:p-Laplace} with 
\eqref{Eq:1:2} in $E_T$.
\end{lemma}
\begin{proof}
Upon using the test function
\[
\z=\frac{(u-k)_+}{(u-k)_++\varep}\vp
\]
 where $\varep>0$, and $\vp\ge0$ satisfying \eqref{Eq:function-space} in the weak formulation \eqref{Eq:int-form} in Section~\ref{S:1:2}, 
 and noticing that the terms involving $\chi_{[u\le0]}$ all vanish due to the choice $k>0$,
the proof then runs similarly to the one in \cite[Lemma~1.1, Chapter~1]{DBGV-mono}.
\end{proof}

\section{Proof of Theorem~\ref{Thm:1:1}}\label{S:Thm-proof}
For  a cylinder $Q_{o}:= Q_R=K_{R}(x_o)\times(t_o-R^{p},t_o)$ with $R\in(0,1)$,
we assume that $Q_{o}\subset E_T$ and  introduce the numbers $\mu^{\pm}$ and $\om$ satisfying
\begin{equation*}
	\mu^+\ge\essup_{Q_{o}} u,
	\quad 
	\mu^-\le\essinf_{Q_{o}} u,
	\quad
	\om\ge\mu^+ - \mu^-.
\end{equation*}
Theorem~\ref{Thm:1:1} is a direct consequence of the following.
\begin{proposition}\label{Prop:int-reg}
	Let $u$ be a locally bounded,  local weak solution to \eqref{Eq:1:1} with 
	 \eqref{Eq:1:2} in $E_T$
	for $p\in(p_*,2)$.
Then there exists a modulus of continuity $\boldsymbol\om(\cdot)$  depending only on the data and $\om$, such that
\[
\essosc_{\widetilde{Q}_r}u\le \boldsymbol\om\Big(\frac{r}{R}\Big), \quad\text{ for all }r\in(0,R)
\]
where $\widetilde{Q}_r\subset Q_o$ is a reference cylinder defined in \eqref{Eq:Q_R}.
Moreover, we can quantify the modulus as
\[
(0,1)\ni\rho\mapsto\boldsymbol\om(\rho)=C \Big(\ln\ln\ln\frac{1}{\rho}\Big)^{-\sig}
\]
for some $\sig\in(0,1)$ depending only on $\{p,N\}$, and some $C>0$ depending on the data and $\om$.
If in addition, the structure conditions \eqref{Eq:5:1:3} -- \eqref{Eq:5:1:4} hold, then 
we can improve the modulus to be
\[
(0,1)\ni\rho\mapsto\boldsymbol\om(\rho)=C \Big(\ln\frac{1}{\rho}\Big)^{-\sig}
\]
where $\sig$ and $C$ have the same dependences as above.
\end{proposition}
\subsection{Proof of Proposition~\ref{Prop:int-reg} begins}\label{S:3:0}
We will first deal with the case when the conditions \eqref{Eq:5:1:3} -- \eqref{Eq:5:1:4} hold
in Sections~\ref{S:3:0} -- \ref{S:modulus}. 
Then the proof for the case without them is analogous, and will be briefly indicated in Section~\ref{S:nonoptimal-modulus}.

Let us assume that $(x_o,t_o)$ coincides with the origin and $\om\le1$ with no loss of generality.
Set $\theta:=(\frac14\om)^{2-p}$ and let $$\rho=R\widetilde{c} \om^{\widetilde{q}}\quad\text{ where } 
\widetilde{q}:=\tfrac{(2-p)(N+p)}{p\lm},$$
and for some $\widetilde{c}\in(0,1)$ to be determined in terms of the data in \eqref{Eq:choice-rho}, such that
\begin{equation}\label{Eq:set-inclusion}
Q_{\rho}(\theta)\subset Q_o\quad\text{ and }\quad \essosc_{Q_{\rho}(\theta)}u\le\om.
\end{equation}
It is important that $\widetilde{q}>0$ and the set inclusion \eqref{Eq:set-inclusion}$_1$ is verified;
this amounts to requiring the super-critical range of $p$ in \eqref{Eq:p-range}.

We first observe that one of following must hold:
\begin{equation}\label{Eq:mu-pm-0}
\mu^+-\tfrac18\om\ge  \tfrac18\om\quad\text{ or }\quad\mu^-+\tfrac18\om\le -\tfrac18\om.
\end{equation}
Otherwise we would have $\mu^+-\mu^-<\frac12\om$, which will be considered later.
Let us suppose the first one, i.e. \eqref{Eq:mu-pm-0}$_1$, holds as the other case is similar.

Consider the following two alternatives:
 \begin{equation}\label{Eq:alternative-int}
\left\{
\begin{array}{cc}
|[u\le\mu^-+\tfrac14\om]\cap Q_{\rho}(\theta)|\le c_o(\tfrac14\om)^{\frac{N+p}p}|Q_{\rho}(\theta)|,\\[5pt]
|[u\le\mu^-+\tfrac14\om]\cap Q_{\rho}(\theta)|> c_o(\tfrac14\om)^{\frac{N+p}p}|Q_{\rho}(\theta)|.
\end{array}\right.
\end{equation}
Here $c_o\in(0,1)$ is determined in Lemma~\ref{Lm:DG:int}.
The following argument consists in showing that the oscillation can be reduced in either case of \eqref{Eq:alternative-int}.

If the first alternative \eqref{Eq:alternative-int}$_1$ holds true, then by Lemma~\ref{Lm:DG:int}
we have
\[
u\ge \mu^-+\tfrac18\om\quad\text{ a.e. in }Q_{\frac12\rho}(\theta),
\]
which in turn yields a reduction of oscillation
\begin{equation}\label{Eq:reduc-osc-1}
\essosc_{Q_{\frac12\rho}(\theta)}u\le\tfrac78\om.
\end{equation}

Before heading to the next stage, it would be helpful to refer to Figure~\ref{Fig-1}
and shed some light on the next step, which is the tricky part.
The choice of $\rho$ is made to accommodate the singularity of $\be(\cdot)$,
through its dependence on $\om$. As such, the cylinder $Q_\rho(\theta)$ is small comparing with $Q_R$. 
Nevertheless, when we go with the second alternative \eqref{Eq:alternative-int}$_2$ next,
only the  singularity of the $p$-Laplacian matters.
Unfortunately, the ``space-time-ratio" of $Q_\rho(\theta)$ is not large enough to accommodate such singularity.
To overcome this difficulty, we will switch the measure information in  \eqref{Eq:alternative-int}$_2$  to
$Q_R(\bar\theta)$ for some $\bar\theta\ll\theta$, followed by an application of the $L^1$ Harnack inequality to propagate it
to the top time level, and then reduce the oscillation  via the expansion of positivity.

\begin{figure}[t]\label{Fig-1}
\centering
\includegraphics[scale=1]{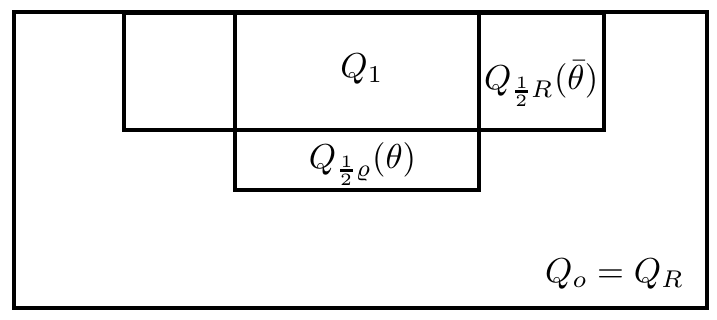}
\caption{Cylinders}
\end{figure}
\subsection{Reduction of oscillation away from zero}\label{S:3:1}
Now we turn our attention to the second alternative \eqref{Eq:alternative-int}$_2$.
Since $\mu^+-\frac14\om\ge\mu^-+\frac14\om$ always holds true, we may rephrase \eqref{Eq:alternative-int}$_2$ as
\[
|[\mu^+-u\ge\tfrac14\om]\cap Q_{\rho}(\theta)|> c_o(\tfrac14\om)^{\frac{N+p}p}|Q_{\rho}(\theta)|=:\al |Q_{\rho}(\theta)|.
\]
Based on this, it is not hard to check that there exists $t_*\in[-\theta\rho^p, -\tfrac12\al\theta\rho^p]$, such that
\begin{equation}\label{Eq:DG:measure:1}
|[\mu^+-u(\cdot, t_*)\ge\tfrac14\om]\cap K_{\rho}  |> \tfrac12\al |K_{\rho}|.
\end{equation}

To proceed, we introduce a new function $v$ defined by
\[
v:=\tfrac18\om-(u-k)_+\quad\text{ with }k=\mu^+ - \tfrac18\om.
\]
Since we have assumed that \eqref{Eq:mu-pm-0}$_1$ holds true, the level $k>0$ and by Lemma~\ref{Lm:sub-solution}
the function $v$ is actually a non-negative, weak super-solution to the parabolic $p$-Laplacian \eqref{Eq:p-Laplace} with \eqref{Eq:1:2}
 in $Q_o$.
The measure information in \eqref{Eq:DG:measure:1}
can be rephrased in terms of $v$ as
\begin{equation}\label{Eq:DG:measure:2}
|[v(\cdot, t_*)\ge\tfrac18\om]\cap K_{\rho}  |> \tfrac12\al |K_{\rho}|\quad\text{ where }\al=c_o(\tfrac14\om)^{\frac{N+p}p}.
\end{equation}

Now we apply Lemma~\ref{Lm:L1} to $v$ in $K_{\frac14R}\subset K_{\frac12R}$ and with $t_*<t<0$.
As a result, we obtain that
\begin{equation}\label{Eq:time:1}
\begin{aligned}
\sup_{t_*<s<t}\int_{K_{\frac14R}}v(x,s)\,\dx&\le\gm\int_{K_{\frac12R}}v(x,t)\,\dx+\gm\Big(\frac{t-t_*}{R^\lm}\Big)^{\frac1{2-p}}\\
&\le\gm\int_{K_{\frac12R}}v(x,t)\,\dx+\gm\Big(\frac{\theta\rho^p}{R^\lm}\Big)^{\frac1{2-p}}\\
&\le\gm\int_{K_{\frac12R}}v(x,t)\,\dx+ \gm\om R^N\Big(\frac{\rho}{R}\Big)^{\frac{p}{2-p}}.
\end{aligned}
\end{equation}
Recall that $\lm=N(p-2)+p$.
We continue to estimate the left-hand side of \eqref{Eq:time:1} from below by using \eqref{Eq:DG:measure:2}:
\begin{equation}\label{Eq:time:2}
\begin{aligned}
\sup_{t_*<s<t}\int_{K_{\frac14R}}v(x,s)\,\dx&\ge \int_{K_\rho}v(x,t_*)\,\dx\ge \tfrac18\om |[v(\cdot, t_*)\ge\tfrac18\om]\cap K_\rho|\\
&\ge \tfrac1{16} c_o(\tfrac14\om)^{1+\frac{N+p}p}|K_\rho|.
\end{aligned}
\end{equation}
In order to combine \eqref{Eq:time:1} and  \eqref{Eq:time:2},
we select $\rho$ to satisfy that
\begin{equation*}
\gm\om R^N\Big(\frac{\rho}{R}\Big)^{\frac{p}{2-p}}=\tfrac1{32} c_o \rho^N(\tfrac14\om)^{1+\frac{N+p}p},
\end{equation*}
that is,
\begin{equation}\label{Eq:choice-rho} 
\rho=R\widetilde{c} \om^{\widetilde{q}}
\quad\text{ where  }\quad\widetilde{q}:=\tfrac{(2-p)(N+p)}{p\lm}\quad\text{ and }\quad
 \widetilde{c}=\Big(\tfrac{1}{2\gm} c_o 4^{-3-\frac{N+p}p}\Big)^{\frac{2-p}\lm}.
\end{equation}
We emphasize again that the super-critical range of $p$ in \eqref{Eq:p-range} is used
in this step, to ensure $\widetilde{q}>0$.
Substituting this choice of $\rho$ in \eqref{Eq:time:1} --  \eqref{Eq:time:2} and combining the two,
we obtain that
\[
\gm\int_{K_{\frac12R}}v(x,t)\,\dx\ge\tfrac1{32} c_o \rho^N(\tfrac14\om)^{1+\frac{N+p}p}\equiv 
\gm \widetilde{c}^{\frac{p}{2-p}} R^N \om^{1+\frac{N+p}\lm}.
\]
Based on this lower bound, we proceed to estimate for some $\sig>0$ that
\begin{align*}
R^N \om^{1+\frac{N+p}{\lm}}&\le \bar\gm\int_{K_{\frac12R}}v(x,t)\,\dx\\
&=\bar\gm \int_{K_{\frac12R}\cap [v\le\sig]}v(x,t)\,\dx+ \bar\gm\int_{K_{\frac12R}\cap [v>\sig]}v(x,t)\,\dx\\
&\le\bar\gm\sig|K_{\frac12R}|+\bar\gm\om|K_{\frac12R}\cap [v(\cdot,t)>\sig]|,
\end{align*}
where $\bar\gm=\widetilde{c}^{\frac{p}{p-2}}$.
After choosing $\sig=\frac1{2\bar\gm}\om^{1+\frac{N+p}{\lm}}$, the above estimate yields that
\begin{equation}\label{Eq:measure-v}
|K_{\frac12R}\cap [v(\cdot,t)>c\om^{1+\frac{N+p}{\lm}}]|\ge c\om^{\frac{N+p}{\lm}}|K_{\frac12R}|,  \quad\text{ where }c=\tfrac1{2\bar\gm},
\end{equation}
for all $t\in(-\tfrac12\al\theta\rho^p, 0)$, which by the definition of $\al$, $\theta$ and $\rho$ amounts to
\begin{equation}\label{Eq:time-interval}
-\bar{c}\om^{\frac{N+p}{p}+(2-p)(1+\frac{N+p}{\lm})} R^p<t<0\quad\text{ where } 
\bar{c}=2^{-\gm}c_o\widetilde{c}^p 
\end{equation}
for some $\gm(p,N)>1$.

The measure information \eqref{Eq:measure-v} for all times in \eqref{Eq:time-interval}
serves as the basis for an application of Lemma~\ref{Lm:expansion} to obtain pointwise estimate.
According to the second part of Lemma~\ref{Lm:expansion} 
(see Section~\ref{S:nonoptimal-modulus} for the application of the first part), 
this measure information \eqref{Eq:measure-v} at each time level $t$ of the  interval
\eqref{Eq:time-interval} yields that
\[
v(\cdot, s)\ge\eta \om^{1+q}\quad\text{ a.e. in } K_{\frac12 R},
\]
where
\[
s=t+\dl_o\big(c\om^{\frac{N+p}{\lm}}\big)^{p+1}\big(c\om^{1+\frac{N+p}{\lm}}\big)^{2-p} R^p, \quad
q=2\tfrac{N+p}{\lm},
\]
$ \dl_o$ is determined in Lemma~\ref{Lm:expansion} and $\eta\in(0,1)$ depends only on the data.
Notice that the power of $\om$ in the above time level $s$ is 
larger than that in the left end of \eqref{Eq:time-interval}.
Consequently, when $t$ ranges in the interval \eqref{Eq:time-interval}, 
we obtain that
\[
v\ge\eta \om^{1+q}\quad\text{ a.e. in } Q_{\frac12 R}(\bar{\theta}),
\]
where
\[
\bar{\theta}=\bar{c}\om^{\bar{q}},\quad\bar{q}=\tfrac{N+p}{p}+(2-p)\big(1+\tfrac{N+p}{\lm}\big). 
\]
 By the definition of $v$, taking $\eta\le\frac18$ this yields a reduction of oscillation for $u$ as
 \begin{equation}\label{Eq:reduc-osc-2}
 \essosc_{Q_{\frac12 R}(\bar{\theta})}u\le (1-\eta\om^q)\om.
 \end{equation}
 
 Now we may combine the reduction of oscillation in \eqref{Eq:reduc-osc-1} and \eqref{Eq:reduc-osc-2},
 irrespective of the alternatives in \eqref{Eq:alternative-int}.
 To this end, recalling the choice of $\rho$ in \eqref{Eq:choice-rho}
 and noticing the simple fact that $2-p+p\widetilde{q}<\bar{q}$, one easily verifies that (cf.~Figure~\ref{Fig-1})
 \[
 Q_{\frac12 R}(\bar{\theta})\cap Q_{\frac12 \rho}( \theta)=K_{\frac12\widetilde\theta R}\times(-\tfrac1{2^p}\bar\theta R^p,0)=:Q_1,
 \]
 where
 \[
\widetilde{\theta}=  \widetilde{c} \om^{\widetilde{q}} \quad\text{ and }\quad \bar{\theta}=\bar{c}\om^{\bar{q}}.
 \]
 Therefore, combining \eqref{Eq:reduc-osc-1} and \eqref{Eq:reduc-osc-2}, we arrive at
  \begin{equation}\label{Eq:reduc-osc-3}
 \essosc_{Q_1}u\le (1-\eta\om^q)\om.
 \end{equation}
 This reduction of oscillation also takes into account the violation of \eqref{Eq:mu-pm-0}.

 \subsection{The induction}\label{S:induction}
 In order to proceed by induction, we introduce
 \[
\om_1= (1-\eta\om^q)\om,\quad R_1=\tfrac12\widetilde{\theta}R,\quad \theta_1=(\tfrac14\om_1)^{2-p},\quad \rho_1=\widetilde{c} R_1\om_1^{\widetilde{q}}.
 \]
To iterate the previous argument, we need to verify the set inclusion
\begin{equation}\label{Eq:set-inclusion-1}
Q_{\rho_1}(\theta_1)\subset Q_1,\quad\text{ i.e. }\quad\theta_1\rho_1^p\le\tfrac1{2^p}\bar\theta R^p,
\end{equation}
which takes the place of \eqref{Eq:set-inclusion}$_1$ in the next stage.
Let us examine the height of the cylinder $Q_{\rho_1}(\theta_1)$, that is,
\[
\theta_1\rho^p_1=(\tfrac14\om_1)^{2-p}(\widetilde{c}R_1\om_1^{\widetilde{q}})^p\le 
(\tfrac14\om)^{2-p}(\tfrac12 \widetilde{c}^2\om^{2\widetilde{q}})^p R^p=2^{p-4}\widetilde{c}^{2p}\om^{2-p+2p\widetilde{q}} R^p.
\]
Here we have used the fact $\om_1\le\om$ to estimate.
Attention is called to the power of $\om$ on the right-hand side as one easily checks that
$2-p+2p\widetilde{q}>\bar{q}$,
using the definition of $\bar{q}$ and $\widetilde{q}$, and the fact that $\lm\in(0,p)$.
Properly making $\widetilde{c}$ smaller if necessary, this implies that
\[
\theta_1\rho^p_1\le 2^{p-4}\widetilde{c}^{2p}\om^{2-p+2p\widetilde{q}} R^p\le \tfrac1{2^p}\bar{c}\om^{\bar{q}}R^p=\tfrac1{2^p}\bar\theta R^p.
\]
Hence the set inclusion \eqref{Eq:set-inclusion-1} is verified.
According to \eqref{Eq:reduc-osc-3}, we then have the intrinsic relation
that takes the place of  \eqref{Eq:set-inclusion}$_2$, i.e.
\[
\essosc_{Q_{\rho_1}(\theta_1)}u\le\om_1,
\]
and the previous argument
can be repeated, such that
 \begin{equation*}
 \essosc_{Q_2}u\le (1-\eta\om_1^q)\om_1\quad\text{ where }
Q_2:=K_{\frac12\widetilde\theta_1 R_1}\times(-\tfrac1{2^p}\bar\theta_1 R_1^p,0)
 \end{equation*}
  and
 \[
\widetilde{\theta}_1=  \widetilde{c} \om_1^{\widetilde{q}}, \qquad \bar{\theta}_1=\bar{c}\om_1^{\bar{q}}
 \]
 for the same parameters $\{\eta,\widetilde{c}, \bar{c}, q, \widetilde{q},\bar{q}\}$.
 
 Now we may proceed by induction and introduce the following notations for $n\in\nn$:
\begin{equation*}
\left\{
\begin{array}{cc}
R_o=R,\quad  R_{n+1}=\tfrac12\widetilde\theta_n R_n,\quad\rho_n=\widetilde{c} R_n\om_n^{\widetilde{q}},\\[5pt]
\quad\theta_n= (\tfrac14\om_n)^{2-p},\quad
\widetilde{\theta}_n=  \widetilde{c} \om_n^{\widetilde{q}}, \quad \bar{\theta}_n=\bar{c}\om_n^{\bar{q}}\\[5pt]
\dsty\om_o=\om,\quad \om_{n+1}= \om_n(1-\eta\om^q_n),\\[5pt]
Q_{n+1}=K_{\frac12\widetilde\theta_n R_n}\times(-\tfrac1{2^p}\bar\theta_n R_n^p,0),\quad Q'_{n}=Q_{\rho_n}(\theta_n).
\end{array}\right.
\end{equation*}
Using the induction we have for all $n\in\nn$ that
\[
Q'_{n}\subset Q_{n}\quad\text{ and }\quad\essosc_{Q_n}\le\om_n.
\]
\subsection{Derivation of the modulus of continuity}\label{S:modulus}
We will derive a modulus of continuity of $u$ inherent in the above oscillation estimate.
For this purpose, we first observe that if a sequence $\{a_n\}$ satisfies
\[
a_n\ge a_{n-1}(1-\eta a_{n-1}^q)\quad\text{ for all }n\ge n_o, \quad\text{ and }\quad a_{n_o}\ge\om_{n_o},
\]
then $a_n\ge \om_n$ for all $n\ge n_o$. As a matter of fact, we may take $a_n= (1+n)^{-\sig} $ with any
 $\sig\in(0,\frac1q)$. In such a case,
the number $n_o$ is determined by $\{q,\eta,\sig\}$ only and hence we may assume $n_o=0$ for simplicity.
Therefore we obtain for all $n\ge0$ that
\[
\essosc_{Q_{n}} u\le \om_n\le\frac{1}{(1+n)^{\sig}}.
\]
Let us take  a number $r\in(0,R)$. 
Since $\bar{c}\om^{\bar{q}}r^p\in(0,\bar\theta_o R^p_{o})$, there must be some $n\ge0$ such that
\begin{equation}\label{Eq:trap-R}
\bar\theta_{n+1}R^p_{n+1}< \bar{c}\om^{\bar{q}}r^p\le\bar\theta_n R^p_{n}.
\end{equation}
From the right-hand side  of \eqref{Eq:trap-R}, we may estimate that
\[
\om^{\widetilde{q}}r=\om^{\frac{\bar{q}}{p}}r \om^{\widetilde{q}-\frac{\bar{q}}{p}}
\le \om_n^{\frac{\bar{q}}p}R_n\om^{\widetilde{q}-\frac{\bar{q}}{p}}\le\om_n^{\widetilde{q}}R_n,
\]
where in the last inequality we have used $\widetilde{q}p<\bar{q}$ and $\om_n\le\om$.
Consequently, we obtain that
\begin{equation}\label{Eq:osc-final}
\essosc_{\widetilde{Q}_{r}} u\le \essosc_{Q_{n} } u\le \om_n\le\frac{1}{(1+n)^{\sig}},
\end{equation}
where
\begin{equation}\label{Eq:Q_R}
 \widetilde{Q}_{r}=K_{\frac12\widetilde{\theta} r}\times(-\tfrac1{2^p}\bar\theta r^p,0).
 \end{equation}
Next, we analyze the left-hand side of \eqref{Eq:trap-R}.
First notice that by the definition of $\om_n$ and assuming $\eta\om^q\le\frac12$ with no loss of generality, 
we can estimate for all $n\ge0$ that
\[
\om_n\ge\frac{\om}{2^n}.
\]
Using this, we may estimate
\[
\bar\theta_{n+1}=\bar{c}\om_{n+1}^{\bar{q}}\ge\bar{c}\Big(\frac{\om}{2^{n+1}}\Big)^{\bar{q}},
\]
and 
\begin{align*}
R_{n+1}=\widetilde{\theta}_nR_n\ge \tfrac12\widetilde{c}\Big(\frac{\om}{2^{n}}\Big)^{\widetilde{q}}R_n
\ge R \prod_{i=0}^{n} \tfrac12\widetilde{c}\Big(\frac{\om}{2^{i}}\Big)^{\widetilde{q}}
=R (\tfrac12\widetilde{c})^{n+1}\frac{\om^{\widetilde{q}(n+1)}}{2^{\widetilde{q}\frac{n(n+1)}2}}.
\end{align*}
Then the  left-hand side of \eqref{Eq:trap-R} gives that
\[
\bar{c}\om^{\bar{q}}r^p>\bar\theta_{n+1}R^p_{n+1}\ge R^p\bar{c}\Big(\frac{\om}{2^{n+1}}\Big)^{\bar{q}}
\bigg[ (\tfrac12\widetilde{c})^{n+1}\frac{\om^{\widetilde{q}(n+1)}}{2^{\widetilde{q}\frac{n(n+1)}2}}\bigg]^p
\]
Then taking logarithm on both sides, a simple calculation will give us a lower bound of $n$ by
\[
n^2\ge\frac{\gm}{1-\ln\om}\ln\frac{R}{r}\quad\text{ for some }\gm=\gm(p,\widetilde{q},\widetilde{c})>1.
\]
Substituting this into \eqref{Eq:osc-final}, we obtain the modulus of continuity
\[
\essosc_{\widetilde{Q}_{r}}u\le \gm (1-\ln\om)^{\frac{\sig}2}\Big(\ln\frac{R}{r}\Big)^{-\frac{\sig}2}\quad\text{ for all }r\in(0,R).
\]
\subsection{Continuity without the comparison principle}\label{S:nonoptimal-modulus}
In this section, we briefly indicate how to modify the arguments in Sections~\ref{S:3:1} -- \ref{S:modulus}
to obtain continuity without the structure conditions \eqref{Eq:5:1:3} -- \eqref{Eq:5:1:4}.

One starts the proof just like Section~\ref{S:3:0} and continues with Section~\ref{S:3:1}.
The present argument departs from \eqref{Eq:measure-v} -- \eqref{Eq:time-interval}.
Now we are only allowed to use the first part of Lemma~\ref{Lm:expansion}.
As a result, we obtain that
\[
v\ge c\eta_o \om ^{q_o}\exp\Big\{-2^{-\frac{\gm}{\om^q}}\Big\}\quad\text{ a.e. in } Q_{\frac12 R}(\bar{\theta}),
\]
where we have set
\[
\bar{\theta}=\bar{c}\om^{\bar{q}},\quad\bar{q}=\tfrac{N+p}{p}+(2-p)q_o,\quad q_o=1+\tfrac{N+p}{\lm}, \quad q=(p+2)\tfrac{N+p}\lm.
\]
This in turn yields the reduction of oscillation
\[
\essosc_{Q_{\frac12 R}(\bar{\theta})}u\le\big(1-\eta(\om)\big)\om\quad\text{ where }\eta=c\eta_o\om ^{q_o}\exp\Big\{-2^{-\frac{\gm}{\om^q}}\Big\}.
\]
Next, this is combined with  the other reduction of oscillation \eqref{Eq:reduc-osc-1}
and the violation of \eqref{Eq:mu-pm-0},
as at the end of Section~\ref{S:3:1}. Then we may employ the induction argument  presented in Section~\ref{S:induction}.
The main difference is the recurrence of $\{\om_n\}$, which now reads, for all $n\in\nn$,
\[
\om_{n+1}=\big(1-\eta(\om_n)\big)\om_n.
\]
An inspection of the recurrence shows that there exists $n_o\in\nn$ depending on the parameters
$\{c,\eta_o,\gm,q,q_o\}$, such that for all $n\ge n_o$,
\[
\om_n\le\big(\ln\ln (n+1)\big)^{-\sig}
\]
for some $\sig\in(0,1)$ depending only on$ \{q,q_o\}$. Reasoning like in Section~\ref{S:modulus},
we obtain an estimate on the modulus of continuity
\[
\essosc_{\widetilde{Q}_{r}}u\le C\Big(\ln\ln\ln\frac{R}{r}\Big)^{-\frac{\sig}2}\quad\text{ for all }r\in(0,R),
\]
for some $C>0$ depending on the data and $\om$, and $\widetilde{Q}_r$ is defined as in \eqref{Eq:Q_R}.
\section{Uniform approximations}\label{S:approx}
Construction of weak solutions to the Stefan problem \eqref{Eq:1:1} generally
consists in first solving regularized versions, deriving {\it a priori} estimates
and obtaining a weak solution via a proper limiting process based on compactness arguments.
Nevertheless, {\it a priori} estimates on the time derivative of solutions are in general not available in this scheme.
Hence, Theorem~\ref{Thm:1:1} does not grant continuity to the so-obtained weak solution automatically.
The purpose of the present section is to demonstrate that
the arguments presented in Section~\ref{S:Thm-proof} can be adapted to 
obtain the equicontinuity of the approximating solutions. Thus the continuity of the limiting function
can be achieved via the Ascoli-Arzela theorem.

To this end, like in \cite[Section~7]{Liao-Stefan},
 let $H_\varep(s)$ be the mollification with $\varep\in(0,1)$, by the standard Friedrichs kernel
supported in $(-\varep,\varep)$ 
of the function
\begin{equation*}
H(s):=\left\{
\begin{array}{cl}
0,\quad&s>0,\\[5pt]
-\nu,\quad& s\le0.
\end{array}\right.
\end{equation*}
Here $\nu$ is from the definition of $\be(\cdot)$. 
Clearly, the function $s\mapsto s+H_\varep(s)$ is an approximation of $\be(s)$.

Consider the regularized version of the Stefan problem \eqref{Eq:1:1}
\begin{equation}\label{Eq:reg}
\begin{aligned}
&\pl_t [u+H_\varep(u)] -\dvg\bl{A}(x,t,u, Du) = 0\quad \text{ weakly in }\> E_T.
\end{aligned}
\end{equation}
For each fixed $\varep$,  
the notion of  solution to \eqref{Eq:reg}
can be defined via a similar integral identity as in Section~\ref{S:1:2}.
A main difference now is that the function space for solutions becomes
\begin{equation*}
\left\{
\begin{array}{cc}
\dsty\int_0^u [1+H'_{\varep}(s)]s\,\d s\in C_{\loc}\big(0,T;L^1_{\loc}(E)\big),\\[5pt]
u\in L^p_{\loc}\big(0,T; W^{1,p}_{\loc}(E)\big).
\end{array}\right.
\end{equation*}
This notion does not require any {\it a priori} knowledge on the time derivative
and is similar to the one for \eqref{Eq:p-Laplace} in \cite[Chapter~II]{DB}, cf.~\eqref{Eq:func-space-1}.
The following theorem can be viewed as a ``cousin" of Theorem~\ref{Thm:1:1}.
\begin{theorem}
Let $\{u_\varep\}$ be a familiy of  weak solutions to \eqref{Eq:reg} with the uniform bound $M$,
 under the structure condition \eqref{Eq:1:2} for $p\in(p_*,2)$.  
Then $\{u_\varep\}$ is locally, equicontinuous in 
$E_T$.
More precisely, there is a modulus of continuity $\boldsymbol\om(\cdot)$,
determined by the data, $\dist_p(\mathcal{K};\,\Gm) $ and $M$, independent of $\varep$, 
such that
	\begin{equation*}
	\big|u_\varep(x_1,t_1)-u_\varep(x_2,t_2)\big|
	\le
	\boldsymbol\om\!\left(|x_1-x_2|+|t_1-t_2|^{\frac1p}\right)+8\varep,
	\end{equation*}
for every pair of points $(x_1,t_1), (x_2,t_2)\in \mathcal{K}$.
If in addition, the structure conditions \eqref{Eq:5:1:3} -- \eqref{Eq:5:1:4}
are satisfied, then there exists $\sig\in(0,1)$ depending only on $p$ and $N$, such that
$$\boldsymbol\om(r)=C \Big(\ln \frac{\dist_p(\mathcal{K};\,\Gm)}{r}\Big)^{-\sig}\quad\text{ for all }r\in\big(0, \dist_p(\mathcal{K};\,\Gm)\big)$$ 
with some $C>0$
 depending on the data and $M$.
\end{theorem}

The subscript $\varep$ will be suppressed from $u$, $\mu^\pm$, $\om$, $\theta$, $\widetilde{\theta}$, etc.
The idea is to adapt the arguments in  Section~\ref{S:Thm-proof} and to determine the quantities, 
such as $\{\eta,\widetilde{c}, \bar{c}, q, \widetilde{q},\bar{q}\}$,
independent of $\varep$.
In this way the reduction of oscillation of $u$ can be achieved just like in  Section~\ref{S:Thm-proof}, independent of $\varep$.
Due to similarities, we will only present a sketchy argument in the following.

Let us first derive the energy estimates similar to Proposition~\ref{Prop:2:1}.
After standard calculations (cf.~\cite[Section~7]{Liao-Stefan}), the  energy estimates for the weak solution $u$ to \eqref{Eq:reg} becomes,
omitting the reference to $x_o$,
\begin{equation}\label{Eq:energy-approx}
\begin{aligned}
	\essup_{t_o-S<t<t_o}&\Big\{\int_{K_R\times\{t\}}	
	\z^p (u-k)_\pm^2\,\dx \pm \int_{K_R\times\{t\}}\int_k^u H_\varep'(s)(s-k)_\pm\,\d s\,\z^p\dx \Big\}\\
	&\quad+
	\iint_{Q_{R,S}}\z^p|D(u-k)_\pm|^p\,\dx\dt\\
	&\le
	\gm\iint_{Q_{R,S}}
		\Big[
		(u-k)^{p}_\pm|D\z|^p + (u-k)_\pm^2|\partial_t\z^p|
		\Big]
		\,\dx\dt\\
			&\quad\pm\iint_{Q_{R,S}}\int_k^u H_\varep'(s)(s-k)_\pm\,\d s |\pl_t\z^p|\, \dx\dt\\
			&\quad
	+\int_{K_R\times \{t_o-S\}} \z^p (u-k)_\pm^2\,\dx\\
	&\quad  \pm \int_{K_R\times\{t_o-S\}}\int_k^u H_\varep'(s)(s-k)_\pm\,\d s\,\z^p\dx.
\end{aligned}
\end{equation}
The three terms containing $H'_\varep$ here play the role of $\Phi_{\pm}$ in Proposition~\ref{Prop:2:1},
which preserve the singularity of $\be(\cdot)$ at the origin as $\varep\to0$.

Next we examine the results in Section~\ref{S:2} in the context of the regularized equation \eqref{Eq:reg}.
Let us consider the energy estimate \eqref{Eq:energy-approx} in the case of super-solution, i.e. $(u-k)_-$. 
The term containing $H'_\varep$ (together with the minus sign in the front) on the left-hand side is non-negative
and hence can be discarded. The first term containing $H'_\varep$ on the right-hand side
is estimated by
\[
\iint_{Q_{R,S}}\int_u^k H_\varep'(s)(s-k)_-\,\d s |\pl_t\z^p|\, \dx\dt
\le \nu\iint_{Q_{R,S}}(u-k)_- |\pl_t\z^p|\,\dx\dt.
\]
The second term containing $H'_\varep$ on the right-hand side is discarded 
because now $\z=0$ on the parabolic boundary of $Q_{R,S}$.
Using these remarks, one can perform the De Giorgi iteration in Lemma~\ref{Lm:DG:int}
and reach the same conclusion.

Furthermore, Lemma~\ref{Lm:expansion} and Lemma~\ref{Lm:L1}
are properties of weak super-solutions to the parabolic $p$-Laplacian,
whereas Lemma~\ref{Lm:sub-solution} continue to hold, if we assume that $k\ge\varep$.

With these preparatory tools at hand, we can start the proof as in Section~\ref{S:3:0}.
A change happens in Section~\ref{S:3:1} when we claim that the function $v$ is 
a non-negative, weak super-solution to the parabolic $p$-Laplacian \eqref{Eq:p-Laplace} with \eqref{Eq:1:2}
 in $Q_o$. In this case, we need to assume that $\frac18\om\ge\varep$ in view of \eqref{Eq:mu-pm-0}
 and the previously mentioned change to Lemma~\ref{Lm:sub-solution}.
As such, the rest of the proof remains unchanged, bearing in mind that violation of this  requirement only contributes
to an extra control of the oscillation by $\om\le8\varep$. 
This $\varep$ can then be incorporated into the derivation of the modulus of continuity.


\begin{thebibliography}{99}
\bibitem{Urbano-14} P. Baroni, T. Kuusi and J.M. Urbano, {\it A quantitative modulus of continuity for the two-phase Stefan problem},
 Arch. Ration. Mech. Anal., {\bf214}(2), (2014),  545--573.
\bibitem{Caff-Evans-83} L.A. Caffarelli and L.C. Evans,  {\it Continuity of the temperature in the two-phase Stefan problem},
 Arch. Ration. Mech. Anal., {\bf81}(3), (1983),  199--220.
\bibitem{Caff-Fried} L.A. Caffarelli and A. Friedman, {\it Continuity of the temperature in the Stefan problem},
 Indiana Univ. Math. J., {\bf28}(1), (1979),  53--70.
\bibitem{DB-82} E. DiBenedetto, {\it Continuity of weak solutions to certain singular parabolic equations},
 Ann. Mat. Pura Appl. (4), {\bf130}, (1982), 131--176.
\bibitem{DB} E. DiBenedetto, ``Degenerate Parabolic 
Equations", Universitext, Springer-Verlag, New York, 1993.  
\bibitem{DBGV-mono} E. DiBenedetto, U. Gianazza and V. Vespri, 
``Harnack's Inequality for Degenerate and Singular Parabolic 
Equations", Springer Monographs in Mathematics, Springer-Verlag, 
New York, 2012.
\bibitem{DBV-95}  E. DiBenedetto  and V. Vespri, {\it On the singular equation $\be(u)_t=\Dl u$},
 Arch. Ration. Mech. Anal., {\bf132}(3), (1995), 247--309. 
\bibitem{Urbano-05} E. Henriques and J.M. Urbano, {\it On the doubly singular equation $\gm(u)_t=\Dl_p u$},
 Comm. Partial Differential Equations, {\bf30}(4-6), (2005),   919--955.
\bibitem{Li-20} Q. Li, {\it On the continuity of solutions to doubly singular parabolic equations},
 Ann. Mat. Pura Appl. (4), {\bf199}(4), (2020),  1381--1429.
\bibitem{Liao-Stefan} N. Liao, {\it On the logarithmic type boundary modulus of continuity for  the Stefan problem}, 
arXiv: 2102.10278.
\bibitem{Urbano-00} J.M. Urbano, {\it Continuous solutions for a degenerate free boundary problem},
 Ann. Mat. Pura Appl. (4), {\bf178}, (2000), 195--224.
\bibitem{Urbano-08} J.M. Urbano,  ``The method of intrinsic scaling, 
A systematic approach to regularity for degenerate and singular PDEs'', 
Lecture Notes in Mathematics, 1930. Springer-Verlag, Berlin, 2008.
\bibitem{Vespri-ACV}
V. Vespri and M. Vestberg, 
{\it An extensive study of the regularity properties of solutions to doubly singular equations}, 
 Adv. Calc. Var., to appear, arXiv:2001.04141.
\end{thebibliography}
\end{document}